\documentclass{amsart}
\usepackage{amsmath,amssymb,latexsym,amsthm,amsfonts,amscd,epic,eepic,flafter,graphicx,verbatim, stmaryrd,mathrsfs,color}
\usepackage[all]{xy}

\def\co{\colon\thinspace}

\author{Yinghua Ai}
\address{School of Mathematical Science, Peking University, Beijing 100871, P.R.China, and Department of Mathematics, Columbia University, MC 4406 2990 Broadway, New York,  NY 10027}
\email{ofey@math.columbia.edu}

\author{Thomas Peters}
\address{Department of Mathematics, Columbia University, MC 4406\\ 2990 Broadway, New York,  NY 10027}
\email{tpeters@math.columbia.edu}

\title{\textbf{the Twisted Floer homology of torus bundles}}
\begin{document}
\newtheorem{theorem}{Theorem}[section]
\newtheorem{lemma}[theorem]{Lemma}
\newtheorem{definition}[theorem]{Definition}
\newtheorem{proposition}[theorem]{Proposition}
\newtheorem*{thm*}{Theorem}
\newtheorem{rem}[theorem]{Remark}

\begin{abstract}
Given a torus bundle $Y$ over  the circle and a cohomology class  $[\omega]\in H^2(Y;\mathbb{Z})$ which evaluates non-trivially on the fiber, we compute the  Heegaard Floer homology of $Y$ with twisted coefficients in the universal  Novikov ring. 
\end{abstract}

\maketitle

\section{Introduction}  
\label{section:intro}
Heegaard Floer homology was introduced by Ozsv\'ath and Szab\'o in \cite{OSzAnn1,OSzAnn2}, it provides powerful invariants for closed oriented $3$--manifolds. There is also a  filtered version, called knot Floer homology, for null-homologous knots \cite{OSzKnot,Ra}.
It turns out that Heegaard Floer homology provides much geometric information.  For instance, it can detect fibrations.  Work of  Ghiggini \cite{Gh} and Ni \cite{Ni1}   shows that knot Floer homology detects fiberedness in knots.  

Turning to closed fibered $3$--manifolds, note that a $3$--manifold which admits a fibration $\pi \co Y\to S^1$ has a canonical Spin$^c$ structure, $\ell$, obtained as the tangents to the fibers of $\pi$. In \cite{OSzSympl}, Ozsv\'ath and Szab\'o prove 

\begin{thm*}
\label{theorem:fiber} Let $Y$ be a closed $3$--manifold which fibers over the circle, with fiber $F$ of  genus $g>1$, and let ${\bf \mathfrak{t}}$ be a Spin$^c$ structure over $Y$ with 
\[
	\langle c_1({\bf \mathfrak{t}}),[F]\rangle=2-2g.
\]  Then for ${\bf \mathfrak{t}}\neq \ell$, we have that
\[
	HF^+(Y,{\bf \mathfrak{t}})=0;
\]
while
\[
	HF^+(Y,\ell)\cong \mathbb{Z}
\] \end{thm*}

In fact, Yi Ni proved a converse to the above statement in \cite{Ni2} which states:

\begin{thm*}
\label{Theorem:Ni} Suppose $Y$ is a closed irreducible $3$--manifold, $F\subset Y$ is a closed connected surface of genus $g>1$.  Let $HF^+(Y,[F],1-g)$ denote the group
\[
	\bigoplus_{ {\bf \mathfrak{s}}\in Spin^c(Y), \langle c_1({\bf \mathfrak{s}}),[F]\rangle=2-2g}HF^+(Y,{\bf \mathfrak{s}}).
\] If $HF^+(Y,[F],1-g)\cong \mathbb{Z}$, then $Y$ fibers over the circle with $F$ as a fiber.
\end{thm*}

For $g=1$, the methods in the above proofs fail.  In his paper \cite{Ni2}, Ni wrote that Ozsv\'ath and Szab\'o suggested to him to use Heegaard Floer homology with twisted coefficients in some Novikov ring in order to extend to the genus $1$ case.  Investigating the $g=1$ case of the aforementioned theorems formed the motivation for this paper.  Much is known, however, about the Floer homology of torus bundles.  For instance, John Baldwin has already computed the untwisted Heegaard Floer homologies of torus bundles with $b_1(Y)=1$ in \cite{Baldwin}.  

In this paper, we use  Heegaard Floer homology with  twisted coefficients in the universal Novikov ring, $\Lambda$, of all formal power series of the form 
\[ \lambda = \sum_{r \in\mathbb{R}}a_r t^r,\,\#\{r\in\mathbb{R}|a_r \neq0,r\leq c\}<\infty \text{ for all } c\in\mathbb{R}\ \] 
where the coefficients  $a_r\in\mathbb{R}$, see \cite[Section 11.1]{MD}. Given a cohomology class $[\omega]\in H^2(Y;\mathbb{Z})$, $\Lambda$ can be given a $\mathbb{Z}[H^1(Y;\mathbb{Z})]$--module structure, and this gives rise to a twisted Heegaard Floer homology $\underline{HF}^+(Y;\Lambda_\omega)$. This version was first defined in \cite[section 10]{OSzAnn1} and can also be derived from the definition of general twisted Heegaard Floer homology in \cite{OSzAnn2}.  We will describe this  group explicitly in Section \ref{subsection:omegatwist}. It is worth noting that  Heegaard Floer homology with twisted  coefficients in a certain Novikov ring has already been studied  extensively in \cite{JM}.  The main theorem we prove in this paper is the following.

\begin{theorem}
\label{theorem:main}
Suppose $Y$ is a closed oriented $3$--manifold which fibers over the circle with torus fiber $F$,  $[\omega] \in H^2(Y;\mathbb{Z})$ is a cohomology class such that $\omega(F)\neq0$.  Then we have an isomorphism of $\Lambda$--modules
\[
\underline{HF}^+(Y;\Lambda_{\omega}) \cong \Lambda.
\]
\end{theorem}

\begin{rem}
 In the setting of Monopole
Floer homology, a corresponding version of this theorem was proved
in \cite[Theorem~42.7.1]{KM}.
In an upcoming paper \cite{AiN}, it is proved that the converse of the above theorem also holds, i.e the twisted Heegaard Floer homology determines  whether an irreducible $3$--manifold is a torus bundle over the circle. 
\end{rem}

This paper is organized as follows.
We provide a review of Heegaard Floer homology with twisted coefficients in section \ref{section:review}, including the most pertinent example, $S^1\times S^2$.  In section \ref{section:exact} we prove a relevant exact triangle for $\omega$--twisted Heegaard Floer homology and prove Theorem \ref{theorem:main}.

\vspace{5pt}\noindent\textbf{Acknowledgements.}
The authors would like to thank Peter Ozsv\'ath for continued guidance and support.
The first author would also like to thank Yi Ni  for suggesting the problem
 and providing some key ideas.

This work was carried out while the first author was an exchange
graduate student at Columbia University, supported by the China Scholarship Council.
 He is grateful to the
Columbia math department for its hospitality.

\section{Review of Twisted Coefficients}
\label{section:review}
We recall the construction of Heegaard Floer homology with twisted coefficients, see \cite{OSzAnn2} and \cite{OSzGenus} for more details.  To a closed oriented $3$--manifold $Y$ we associate a pointed Heegaard diagram $(\Sigma,\mathbf{\alpha}, {\bf \beta}, z)$, where $\Sigma$ is an an oriented surface of genus $g\geq1$ and ${\bf \alpha} = \{\alpha_1,...,\alpha_g\}$ and ${\bf \beta} = \{\beta_1,...\beta_g\}$ are sets of attaching circles (assumed to intersect transversely) for the two handlebodies in the Heegaard decomposition.  These give a pair of transversely intersecting $g$-dimensional tori $\mathbb{T}_\alpha = \alpha_1\times\cdots\times\alpha_g$ and $\mathbb{T}_\beta=\beta_1\times\cdots\times\beta_g$ in the symmetric product $Sym^g(\Sigma)$.  Recall that the basepoint $z$ gives a map $\mathfrak{s}_z:\mathbb{T}_\alpha\cap \mathbb{T}_\beta\to Spin^c(Y)$.  Given a Spin$^c$ structure ${\bf \mathfrak{s}}$ on Y, let $\mathfrak{S}\subset \mathbb{T}_\alpha\cap \mathbb{T}_\beta$ be the set of intersection points ${\bf x}\in \mathbb{T}_\alpha\cap \mathbb{T}_\beta$ such that $\mathfrak{s}_z({\bf x})={\bf \mathfrak{s}}$.

Given intersection points ${\bf x}$ and ${\bf y}$ in $\mathbb{T}_\alpha\cap \mathbb{T}_\beta$, let $\pi_2({\bf x},{\bf y})$ denote the set of homotopy classes of Whitney disks from ${\bf x}$ to ${\bf y}$.  There is always a natural map from $\pi_2({\bf x},{\bf x})$ to $H^1(Y;\mathbb{Z})$ obtained as follows: each $\phi\in \pi_2({\bf x},{\bf x})$ naturally gives rise to an associated two-chain in $\Sigma$ whose boundary is a collection of circles among the ${\bf \alpha}$ and ${\bf \beta}$ curves.  We then close off this two-chain by gluing copies of the attaching disks for the handlebodies in the Heegaard diagram for $Y$.  The Poincar\'e dual of this two-cycle is the associated element of $H^1(Y;\mathbb{Z})$. 

Given a Spin$^c$ structure ${\bf \mathfrak{s}}$ on $Y$ and an ${\bf \mathfrak{s}}$-admissible Heegaard diagram $(\Sigma,{\bf \alpha}, {\bf \beta},z)$ for $Y$, an \emph{additive assignment} is a collection of maps
\[	
	A=\{A_{{\bf x},{\bf y}}:\pi_2({\bf x},{\bf y})\to H^1(Y;\mathbb{Z})\}_{{\bf x,y}\in \mathfrak{S}}
\]
so that:
\begin{itemize}
\item when ${\bf x}={\bf y},$ $A_{{\bf x},{\bf x}}$ is the canonical map from $\pi_2({\bf x},{\bf x})$ onto $H^1(Y;\mathbb{Z})$ defined above.
\item A is compatible with splicing in the sense that if ${\bf x,y,u}\in \mathfrak{S}$ then for each $\phi_1\in\pi_2({\bf x},{\bf y})$ and $\phi_2\in\pi_2({\bf y},{\bf u})$, we have that $A(\phi_1\ast\phi_2) = A(\phi_1)+A(\phi_2)$.
\item $A_{\mathbf{x},\mathbf{y}}(S*\phi)=A_{\mathbf{x},\mathbf{y}}(\phi)$ for $S\in\pi_2(Sym^g(\Sigma_g))$.
\end{itemize}  Additive assignments may be constructed with the help of a \emph{complete system of paths} as described in \cite{OSzAnn2}.

We write elements in the group-ring $\mathbb{Z}[H^1(Y;\mathbb{Z})]$ as finite formal sums\\ $\sum _{g\in H^1(Y;\mathbb{Z})} n_g\cdot e^g$ for $n_g\in\mathbb{Z}$.  The \emph{universally twisted Heegaard Floer complex}, $\underline{CF}^{\infty}(Y,{\bf \mathfrak{s}};\mathbb{Z}[H^1(Y;\mathbb{Z})],A)$, is the free $\mathbb{Z}[H^1(Y;\mathbb{Z})]$--module on generators $[{\bf x},i]$ for ${\bf x}\in \mathfrak{S}$ and $i\in\mathbb{Z}$.  The differential, $\underline{\partial}^{\infty}$, is given by
\[
	\underline{\partial}^{\infty}[x,i] = \sum_{ {\bf y}\in\mathfrak{S} }\sum_{ \genfrac{}{}{0pt}{}{ \phi\in\pi_2({\bf x},{\bf y})}{\mu(\phi)=1} } \#\widehat{\mathcal{M}}(\phi)\cdot e^{A(\phi)}\otimes[{\bf y},i-n_z(\phi)].
\]  Here $\mu(\phi)$ denotes the Maslov index of $\phi$, the formal dimension of the space $\mathcal{M}(\phi)$ of holomorphic representatives, $n_z(\phi)$ denotes the intersection number of $\phi$ with the subvariety $\{z\}\times Sym^{g-1}(\Sigma) \subset Sym^g(\Sigma)$, and $\widehat{\mathcal{M}}(\phi)$ denotes the quotient of $\mathcal{M}(\phi)$ under the natural action of $\mathbb{R}$.  Just as in the untwisted setting, this complex admits a $\mathbb{Z}[U]$--action via $U:[\mathbf{x},i]\mapsto[\mathbf{x},i-1]$.  This gives rise to variants $\underline{CF}^+,\underline{CF}^-$, and $\underline{\widehat{CF}}$, denoted collectively as $\underline{CF}^\circ$.
The homology groups of this complexes are the universally twisted Heegaard Floer homology groups $\underline{HF}^\circ$.

 More generally, given any $\mathbb{Z}[H^1(Y;\mathbb{Z})]$--module $M$, we may form Floer homology groups with coefficients in $M$ by taking $\underline{HF}^\circ(Y,{\bf \mathfrak{s}};M)$ as the homology of the complex
\[	
	\underline{CF}^\circ(Y,{\bf \mathfrak{s}};M,A):=\underline{CF}^\circ(Y,{\bf \mathfrak{s}};\mathbb{Z}[H^1(Y;\mathbb{Z})],A)\otimes_{\mathbb{Z}[H^1(Y;\mathbb{Z})]}M.
\]  For instance, by taking $M=\mathbb{Z}$, thought of as being a trivial $\mathbb{Z}[H^1(Y;\mathbb{Z})]$--module, one recovers the ordinary untwisted Heegaard Floer homology, $CF^\circ(Y,\mathfrak{s})$.

In \cite{OSzAnn1}, it is proved that  the homologies defined above are independent of the choice of additive assignment $A$ and are topological invariants of the pair  $(Y,\mathfrak{s})$.  As in the untwisted setting, these groups are related by long exact sequences 

\begin{displaymath}
\xymatrix{	\cdots \ar[r] & \underline{\widehat{HF}}(Y,\mathfrak{s};M) \ar[r] & \underline{HF}^+(Y,\mathfrak{s};M) \ar[r]^U & \underline{HF}^+(Y,\mathfrak{s};M) \ar[r] & \cdots }
\end{displaymath}
and 
\begin{displaymath}
\xymatrix{	\cdots \ar[r] & \underline{HF}^-(Y,\mathfrak{s};M) \ar[r]^\iota & \underline{HF}^\infty(Y,\mathfrak{s};M) \ar[r]^\pi & \underline{HF}^+(Y,\mathfrak{s};M) \ar[r] & \cdots }
\end{displaymath}

Of course, the chain complex $\underline{CF}^\circ(Y,\mathfrak{s};M)$ is obtained from the chain complex in the universally twisted case, $\underline{CF}^\circ(Y,\mathfrak{s};\mathbb{Z}[H^1(Y;\mathbb{Z})])$, by a change of coefficients and hence the corresponding homology groups are related by the universal coefficients spectral sequence (\cite{CE}).

\subsection{$\omega$--twisted Heegaard Floer homology}
\label{subsection:omegatwist}

In this section we briefly recall the notion of $\omega$--twisted Heegaard Floer homology, following 
\cite{OSzAnn1,OSzGenus}.

The universal Novikov ring is defined to be
  $$
  \Lambda = \left\{\sum_{r \in \mathbb{R}}a_r t^r\bigg|a_r\in\mathbb{R},
  \;\#\{a_r|a_r\ne0, r \leq c\}<\infty\quad \text{\rm for all $c\in\mathbb R$}
  \right\},
  $$
endowed with  the following  multiplication law which makes it into a field:
$$
(\sum_{r \in \mathbb{R}}a_rt^r ) \cdot (\sum_{r \in \mathbb{R}}b_rt^r)=\sum_{r \in \mathbb{R}}(\sum_{s \in \mathbb{R}}a_sb_{r-s})t^r.
$$
Furthermore, by fixing a cohomology class $[\omega]\in H^2(Y;\mathbb{R})$ we can give $\Lambda$ a\\ $\mathbb{Z}[H^1(Y;\mathbb{Z})]$--module structure
via the ring homomorphism induced by:
 \[
 \begin{array}{ccl}
H^1(Y;\mathbb{Z})&\to &\mathbb{R}\\
  \gamma  &\mapsto&
\int_{Y} \gamma \wedge \omega
\end{array}.
  \]
When we are interested in its $\mathbb{Z}[H^1(Y;\mathbb{Z})]$--module structure, we denote it 
as $\Lambda_{\omega}$.  This $\mathbb{Z}[H^1(Y;\mathbb{Z})]$--module structure gives rise to a  twisted Heegaard Floer
homology $\underline{HF}^+(Y;\Lambda_{\omega})$, which we refer to as
 $\omega$--\emph{twisted Heegaard Floer homology}. More concretely,
it can be defined as follows. Choose a weakly admissible pointed  Heegaard diagram
$(\Sigma,\mbox{\boldmath${\alpha}$}, \mbox{\boldmath${\beta}$},z)$
for $Y$ and fix a 2--cocycle representative $\omega\in[\omega]$. Every Whitney disk $\phi$ in $\mathrm{Sym}^g(\Sigma)$
(for $\mathbb{T}_\alpha$ and $\mathbb{T}_\beta$) gives rise to a
two-chain $[\phi]$ in $Y$ by coning off partial $\alpha$ and $\beta$ circles with gradient trajectories in the $\alpha$ and $\beta$ handlebodies. The evaluation of $\omega$ on $[\phi]$
depends only on the homotopy class of $\phi$ and is denoted
$\int_{[\phi]}\omega$ (or sometimes $\omega([\phi]))$. The $\omega$--twisted chain complex
$\underline{CF}^+(Y; \Lambda_{\omega})$ is the free
$\Lambda$--module generated by $[x,i]$ with $x \in
\mathbb{T}_\alpha \cap \mathbb{T}_\beta$ and integers $i \geq 0$,
endowed with the following differential:
\[
    \underline{\partial}^+[{\bf x},i] = \sum_{{\bf y}\in \mathbb{T}_\alpha\cap \mathbb{T}_\beta}\sum_{ \{\phi\in\pi_2({\bf x}, {\bf y})|\mu(\phi)=1\}} \#\widehat{\mathcal{M}}(\phi)[{\bf y},i-n_z(\phi)]\cdot t^{\int_{[\phi]}\omega}.
\]
Its homology is the $\omega$-twisted Heegaard Floer homology
$\underline{HF}^+(Y;\Lambda_{\omega})$.
 Notice that this group is both a module for $\mathbb{Z}[H^1(Y;\mathbb{Z})]$
  and a module for $\Lambda$.  Notice that although the differential depends on the choice of 2--cocycle representative $\omega\in[\omega]$, the isomorphism class of the chain complex only depends on the cohomology class.  The advantage of using this viewpoint is that we avoid altogether the notion of an ``additive assignment".  It is easy to see, however, that the complex defined above is isomorphic to one obtained by chosing an additive assignment and then tensoring with the $\mathbb{Z}[H^1(Y;\mathbb{Z})]$-module $\Lambda_{\omega}$.
 
Suppose $W:Y_1 \to Y_2$ is a $4$--dimensional cobordism from $Y_1$ to $Y_2$ given by a single $2$--handle addition and we have a cohomology class  $[\omega]\in H^2(W;\mathbb{R})$.  Then there is an associated 
 Heegaard triple $(\Sigma,{\bf \alpha}, {\bf \beta}, {\bf \gamma}, z)$ and 
 $4$--manifold $X_{\alpha\beta\gamma}$ representing $W$ minus a 
 one complex.  Similar to before, a Whitney triangle $\psi\in\pi_2({\bf x},{\bf y},{\bf w})$ determines a two-chain in $X_{\alpha\beta\gamma}$ on which we may evaluate a representative, $\omega\in[\omega]$.  As before, this evaluation depends only on the
  homotopy class of $\psi$ and is denoted by $\int_{[\psi]}\omega$.  This gives rise to  a $\Lambda$--equivariant map
\[	
	\underline{F}^+_{W;\omega} \co \underline{HF}^+(Y_1;\Lambda_{\omega|_{Y_1}})\to \underline{HF}^+(Y_2;\Lambda_{\omega|_{Y_2}})
\]defined on the chain level by
\[
\underline{f}^+_{W;\omega}[{\bf x},i] = \sum_{{\bf y}\in \mathbb{T}_\alpha \cap \mathbb{T}_\gamma} \sum_{ \genfrac{}{}{0pt}{}{\psi\in\pi_2({\bf x},\Theta,{\bf y})}{\mu(\psi)=0} } \#\mathcal{M}(\psi) [{\bf y},i-n_z(\psi)]\cdot t^{\int_{[\psi]}\omega}.
\] where $\Theta \in \mathbb{T}_\beta \cap \mathbb{T}_\gamma$ represents a top 
dimensional generator for the Floer homology 
$HF^{\leq 0}(Y_{\beta \gamma}) \cong \wedge^{\ast} H^1(Y_{\beta \gamma})\otimes\mathbb{Z}[U]$
 of the $3$--manifold determined by the Heegaard diagram $(\Sigma,{\bf \beta},{\bf \gamma},z)$, which is
a connected sum $\#^{g-1}(S^1 \times S^2)$.  These maps may be decomposed as
a sum of maps, $$\underline{F}^+_{W;\omega} = \sum_{\mathfrak{s}\in Spin^c(W)}\underline{F}^+_{W,\mathfrak{s};\omega},$$ according to Spin$^c$ equivlence classes of triangles, 
just as in the untwisted setting.  This can be extended to arbitrary (smooth, connected) cobordisms from $Y_1$ to $Y_2$ as in \cite{OSzFour}.  These maps also satisfy a \emph{composition law}: if $W_1$ is a cobordism from $Y_1$ to $Y_2$ and $W_2$ is a cobordism from $Y_2$ to $Y_3$, and we equip $W_1$ and $W_2$ with Spin$^c$ structures $\mathfrak{s}_1$ and $\mathfrak{s}_2$ respectively (whose restrictions agree over $Y_2$), then putting $W=W_1\#_{Y_2}W_2$, for any $[\omega]\in H^2(W;\mathbb{R})$ we have 
\[ \underline{F}^+_{W_2,\mathfrak{s}_2;\omega|_{W_2}} \circ \underline{F}^+_{W_1,\mathfrak{s}_1;\omega|_{W_1}} = \sum_{ \{\mathfrak{s}\in Spin^c(W):\mathfrak{s}|_{W_i}=\mathfrak{s}_i\}} \pm \underline{F}^+_{W,\mathfrak{s};\omega}. \]

\subsection{Example: $S^1\times S^2$}
\label{subsection:s1xs2}

In this section we calculate twisted Heegaard Floer homologies of $S^1\times S^2$.  We start with the universally twisted version $\underline{\widehat{HF}}(S^1\times S^2; \mathbb{Z}[t,t^{-1}])$, where we have identified $\mathbb{Z}[H^1(S^1\times S^2;\mathbb{Z})] \cong \mathbb{Z}[t,t^{-1}]$, the ring of Laurent polynomials.  $S^1\times S^2$ has a standard genus one Heegaard decomposition $(\Sigma,\alpha,\beta)$ where $\alpha$ is a homotopically nontrivial embedded curve and $\beta$ is an isotopic translate.  For simplicity, we only compute $\underline{\widehat{HF}}$.  We make the diagram weakly admissible for the unique torsion Spin$^c$ structure  $\mathfrak{s}_0$ by introducing cancelling pairs of intersection points between $\alpha$ and $\beta$.  This gives a pair of intersection points $x^+$ and $x^-$.  We next need an additive assignment.  Notice there is an obvious periodic domain consisting of a pair of (nonhomotopic) disks $D_1$ and $D_2$ connecting $x^+$ and $x^-$.  When capped off, the periodic domain gives a sphere representing a generator of $H_2(S^1\times S^2;\mathbb{Z}) \cong \mathbb{Z}$.  Hence taking Poincar\'e dual we recover a generator of $H^1(S^1\times S^2;\mathbb{Z}) \cong \mathbb{Z}$.  Identifying $H^1(S^1\times S^2;\mathbb{Z}) \cong \mathbb{Z}$, an additive assignment must assign $1$ to this domain.  One way this can be done is by assigning $1$ to $D_1$ and $0$ to $D_2$.  Hence we see that the complex $\underline{\widehat{CF}}(S^1\times S^2;\mathbb{Z}[t,t^{-1}])$ is just 
\begin{displaymath}
\xymatrix{0 \ar[r] & \mathbb{Z}[t,t^{-1}] \ar[r]^{1-t} & \mathbb{Z}[t,t^{-1}] \ar[r] & 0}.
\end{displaymath}
Here, the first copy of $\mathbb{Z}[t,t^{-1}]$ corresponds to $x^+$ and the second corresponds to $x^-$.  This complex has homology $\mathbb{Z}$, with trivial $\mathbb{Z}[t,t^{-1}]$--action.  If we keep track of gradings, we get the universally twisted Floer homology
\[
\underline{\widehat{HF}}(S^1\times S^2;\mathbb{Z}[t,t^{-1}]) \cong \mathbb{Z}_{(-\frac{1}{2})}
\] supported only in the torsion Spin$^c$ structure $\mathfrak{s}_0$.

Now let's turn to an $\omega$-twisted example.  We can view $S^1\times S^2$ as $0$-surgery on the unknot in $S^3$.  Put $\mu$ a meridian for the unknot.  Then $\mu$ defines a curve, also denoted $\mu$, in $S^1\times S^2$.  Put $[\omega]=d\cdot PD[\mu]$ for an integer $d$.  The complex $\underline{\widehat{CF}}(S^1\times S^2;\Lambda_{\omega})$ is
\begin{displaymath}
\xymatrix{0 \ar[r] &\Lambda \ar[r]^{t^c(1-t^d)} & \Lambda \ar[r] &0}.
\end{displaymath} for some $c\in\mathbb{R}$.
Notice when $d \neq 0$,  $(1-t^d)$ is invertible in $\Lambda$. Hence
$$
\underline{\widehat{HF}}(S^1\times S^2;\Lambda_{\omega})=\begin{cases} 0 &\text{when $d \neq 0$}\\
\Lambda \oplus \Lambda&\text{when $d=0$}
\end{cases}.
$$

As a final example, we prove a proposition regarding embedded $2$--spheres in $3$--manifolds and $\omega$-twisted coefficients.  We begin with the following

\begin{lemma}
\label{lemma:connectsum}
Let $Y_1$ and $Y_2$ be a pair of closed oriented $3$--manifolds and fix cohomology classes $[\omega_i]\in H^2(Y_i;\mathbb{Z})$.  By the Mayer-Vietoris sequence we get a corresponding cohomology class $\omega_{1}\#\omega_{2}\in H^2(Y_1\#Y_2;\mathbb{Z}) \cong H^2(Y_1;\mathbb{Z})\oplus H^2(Y_2;\mathbb{Z})$.  Then we have an isomorphism of $\Lambda$--modules
\[
\underline{\widehat{HF}}(Y_1\#Y_2;\Lambda_{\omega_1\#\omega_2}) \cong \underline{\widehat{HF}}(Y_1;\Lambda_{\omega_{1}})\otimes_{\Lambda} \underline{\widehat{HF}}(Y_2;\Lambda_{\omega_{2}}).
\]

\end{lemma}
\begin{proof}
This follows readily from the methods of proof of \cite[Proposition 6.1]{OSzAnn2} and the fact that $\Lambda$ is a field.
\end{proof}

This allows us to prove

\begin{proposition}
\label{proposition:spheres}
Let $S \subset Y^3$ be an embedded nonseparating $2$--sphere in a $3$--manifold $Y$.  Suppose $[\omega] \in H^2(Y;\mathbb{Z})$ is a cohomology class such that $\omega([S])\neq0$.  Then $\underline{HF}^+(Y;\Lambda_\omega)=0$.
\end{proposition}
\begin{proof}

Just as in the untwisted theory, $\underline{HF}^+(Y;M)$ vanishes if and only if $\underline{\widehat{HF}}(Y;M)$ vanishes, so it suffices to show that $\underline{\widehat{HF}}(Y;\Lambda_\omega)=0$.  Notice that $Y$ contains an $S^1\times S^2$ summand in its prime decomposition.  Hence $Y\cong S^1\times S^2\#Y'$ for some $3$--manifold $Y'$.  Now $\omega\in H^2(Y;\mathbb{Z})\cong H^2(S^1\times S^2;\mathbb{Z})\oplus H^2(Y';\mathbb{Z})$ corresponds to classes $\omega_1\in H^2(S^1\times S^2;\mathbb{Z})$ and $\omega_2\in H^2(Y';\mathbb{Z})$ with $\omega_1([S])\neq0$.  We already know that $\underline{\widehat{HF}}(S^1\times S^2;\Lambda_{\omega_{1}})=0$ from the above calculation, so the proposition follows from Lemma                          \ref{lemma:connectsum}.
\end{proof}

\section{Exact sequence for $\omega$--twisted Floer homology}
\label{section:exact}
In this section we first prove a long exact sequence for the $\omega$--twisted Heegaard Floer homologies and then use it to prove Theorem \ref{theorem:main}. It is interesting to notice that there is a similar exact sequence in Monopole Floer homology with local coefficients, see \cite[Section 5]{KMOS}.  Our proof is a slight modification of the proof of the usual surgery exact sequence in Heegaard Floer homology.

Let $K \subset Y$ be framed knot in a $3$--manifold $Y$ with framing $\lambda$ and meridian $\mu$. Given an integer $r$, let $Y_r(K)$ denote  the $3$--manifold obtained from $Y$ by doing Dehn surgery along the knot $K$ with framing $\lambda+r\mu$. Let $\eta \subset Y-nbd(K)$  be a closed curve in the knot complement. Then for any integer $r$, $\eta \subset Y-nbd(K) \subset Y_r(K)$ is a closed curve in the surgered manifold $Y_r(K)$, we denote its Poincar\'e dual by $[\omega_r] \in H^2(Y_r(K);\mathbb{Z})$.  Also note that $\eta \times I$ represents a relative homology class in the corbordisms $W_0 \co Y \to Y_0(K)$,  $W_1 \co Y_0(K) \to Y_1(K)$ and $W_2 \co Y_1(K) \to Y$.  So as in Section \ref{subsection:omegatwist} it gives rise to  homomorphisms between $\omega$--twisted Floer homologies:
$$\underline{F}^+_{W_0;PD(\eta \times I)} \co \underline{HF}^+(Y;\Lambda_{\omega}) \to \underline{HF}^+(Y_0(K);\Lambda_{\omega_0}),$$
$$\underline{F}^+_{W_1;PD(\eta \times I)} \co \underline{HF}^+(Y_0(K);\Lambda_{\omega_0}) \to \underline{HF}^+(Y_1(K);\Lambda_{\omega_1}),$$
and 
$$\underline{F}^+_{W_2;PD(\eta \times I)} \co \underline{HF}^+(Y_1(K);\Lambda_{\omega_1}) \to \underline{HF}^+(Y;\Lambda_{\omega}),$$
where $\omega=PD(\eta) \in H^2(Y;\mathbb{Z})$.
We denote the corresponding maps on the chain level by $\underline{f}^+_{W_0;PD(\eta \times I)}$, $\underline{f}^+_{W_1;PD(\eta \times I)}$ and $\underline{f}^+_{W_2;PD(\eta \times I)}$ respectively.

\begin{theorem}
\label{theorem:exact} The maps above form an exact sequence of $\Lambda$-modules:

\begin{equation}\label{Seq:01}
\begin{xymatrix}{ \underline{HF}^+(Y;\Lambda_\omega) \ar[rr]
&&\underline{HF}^+(Y_0(K);\Lambda_{\omega_0}) \ar[dl]\\
&\underline{HF}^+(Y_1(K);\Lambda_{\omega_1})\ar[ul]&}
\end{xymatrix}
\end{equation}
 Furthermore, analogous exact sequences hold for ``hat" versions as well.
\end{theorem}

\begin{proof} Find a Heegaard diagram $(\Sigma, \{\alpha_1,\cdots, \alpha_g\},\{\beta_1,\cdots, \beta_g\},z)$ compatible with the knot $K$, i.e $K$ lies in the handlebody specified by the $\beta$ curves and $\beta_1$ is a meridian for $K$.  For each $i \geq 2$ let $\gamma_i,\delta_i$ be exact Hamiltonian isotopies of $\beta_i$.  Let $\gamma_1=\lambda$, $\delta_1=\lambda+\mu$  be the $0$-framed and $1$-framed longitude of the knot $K$, respectively.  We assume the Heegaard quadruple $(\Sigma, \alpha,\beta,\gamma,\delta,z)$ is weakly admissible in the sense of \cite{OSzAnn1}.  It is easy to see that $Y_{\alpha \beta}=Y$, $Y_{\alpha \gamma}=Y_0(K)$, $Y_{\alpha \delta}=Y_1(K)$, and $Y_{\beta \gamma} \cong Y_{\gamma \delta} \cong Y_{\beta \delta} \cong \#^{g-1}S^2 \times S^1$.

Following \cite{OSzDouble}, we define a map
 \[h_1\co \underline{CF}^+(Y;\Lambda_{\omega}) \to \underline{CF}^+(Y_1(K);\Lambda_{\omega_1})\] 
 by counting holomorphic rectangles:
 \[
h_1([\mathbf{x},i])=\sum_{\mathbf{w} \in \mathbb{T}_\alpha \cap \mathbb{T}_\delta}\sum_{ \genfrac{}{}{}{0pt}{\varphi \in \pi_2(\mathbf{x},\Theta_{\beta \gamma},\Theta_{ \gamma \delta},\mathbf{w})}\mu(\varphi)=-1}\# \mathcal{M}(\varphi)[\mathbf{w},i-n_z(\varphi)]t^{\int_{[\varphi]}PD(\eta \times I)}
\]
Similarly we define $h_2\co \underline{CF}^+(Y_0(K);\Lambda_{\omega_0}) \to \underline{CF}^+(Y;\Lambda_{\omega})$ and\\ $h_3\co \underline{CF}^+(Y_1(K);\Lambda_{\omega_1}) \to \underline{CF}^+(Y_0(K);\Lambda_{\omega_0})$.

We claim that  $h_1$ is a nullhomotopy of $\underline{f}^+_{W_1;PD(\eta \times I)}\circ\underline{f}^+_{W_0;PD(\eta \times I)}$.  To see this, we consider the  moduli spaces of holomorphic rectangles of Maslov index $0$.  This moduli spaces can have 6 kinds of ends: 
\begin{enumerate}
\item  splicing holomorphic discs at one of its 4 corners, and
\item   splicing two holomorphic triangles.  Triangles may be spliced in two ways: one triangle for  $X_{\alpha \beta\gamma}$ and one triangle for $X_{\alpha \gamma \delta}$, or one triangle for $X_{\alpha \beta\delta}$ and one triangle for $X _{\beta \gamma \delta}$
\end{enumerate}
Notice $PD(\eta \times I)$ is $0$ when restricted to the corners $Y_{\beta \gamma}$ and $Y_{\gamma \delta}$: in fact, we can make $\eta\times I$ disjoint from these manifolds since $\eta$ may be pushed completely into the $\alpha$-handlebody, $U_\alpha$, by cellular approximation (see figure \ref{fig:4m}).

\begin{figure}[htbp] \begin{center}
\input{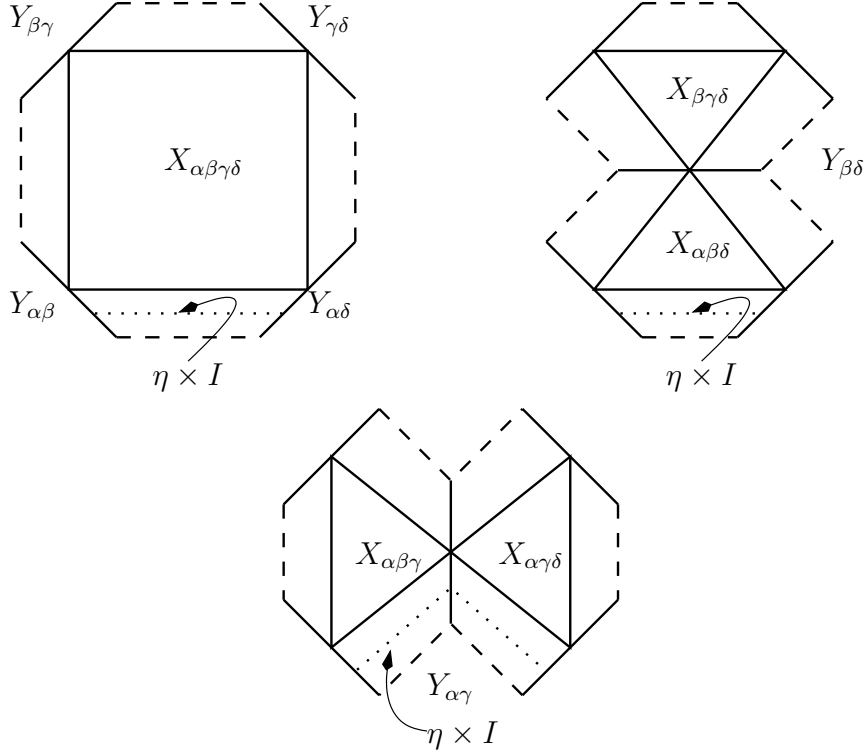}
\vspace*{-2mm} \caption{\label{fig:4m} Schematics of the 4-manifold $X_{\alpha\beta\gamma\delta}$ and its decompositions.  }
\end{center}
\vspace*{-5mm}\end{figure}

 This imples that \[\underline{CF}^+(Y_{\beta\gamma};\Lambda_{PD(\eta \times I)|_{Y_{\beta \gamma}}})\cong CF^+(Y_{\beta\gamma})\otimes_\mathbb{Z}\Lambda\] and all differentials are trivial (informally, we are using an ``untwisted" count). For the end coming from splicing two holomorphic triangles, one for $X_{\alpha \beta\delta}$ and one for $X _{\beta \gamma \delta}$, it is also true that $PD(\eta \times I)$ is 0 when restricted to the 4-manifold $X_{\beta \gamma \delta}$ (again, since $\eta$ may be pushed completely into $U_\alpha$). Therefore we are  counting holomorphic triangles in $X_{\beta \gamma \delta}$ without twisting. In \cite{OSzAnn2} it is shown that the untwisted counting of holomorphic triangles in $X_{\beta \gamma \delta}$ is zero.  This leaves three terms remaining: 
 \begin{enumerate}
\item  splicing a disc at corner $Y_{\alpha \beta}$ counted with twisting by $PD(\eta \times I)|_{Y_{\alpha \beta}}=[\omega]$.  This corresponds to $h_1 \circ \partial$.
\item  splicing a disc at corner $Y_{\alpha \delta}$ counted with twisting by $PD(\eta \times I)|_{Y_{\alpha \delta}}=[\omega_1]$.  This corresponds to $\partial \circ h_1$, and
\item splicing two holomorphic triangles from  $X_{\alpha \beta\gamma}$ and $X_{\alpha \gamma \delta}$ counted with twisting by $PD(\eta \times I)$.  This corresponds to $\underline{f}^+_{W_1;PD(\eta \times I)}\circ\underline{f}^+_{W_0;PD(\eta \times I)}$.
\end{enumerate}
From the fact that the moduli space must have total end zero, it is clear that the sum of the above $3$ terms are zero, i.e  $h_1$ is a homotopy connecting 
$\underline{f}^+_ {W_1;PD(\eta \times I)}\circ\underline{f}^+_{W_0;PD(\eta \times I)}$ to the zero map. This shows that $\underline{F}^+_{W_1;PD(\eta \times I)}\circ\underline{F}^+_{W_0;PD(\eta \times I)}=0$ on the homology level.  The same argument shows that $\underline{F}^+_{W_2;PD(\eta \times I)}\circ\underline{F}^+_{W_1;PD(\eta \times I)}=0$ and $\underline{F}^+_{W_0;PD(\eta \times I)}\circ\underline{F}^+_{W_2;PD(\eta \times I)}=0$ as well.
 
 At last we prove that  Sequence (\ref{Seq:01}) is exact.  Using the homological algebra method in \cite{OSzDouble} we need to show that $h \circ\underline{f}^{+}+\underline{f}^{+} \circ h$ is homotopic to the identity map.  This can be done by counting holomorphic pentagons and noticing that $PD(\eta \times I)$ is zero when restricting to $Y_{\beta \gamma}$,$Y_{ \gamma \delta}$,$Y_{\delta \beta'}$,$X_{\beta \gamma \delta}$, $X_{\gamma \delta \beta'}$ and $X_{\beta \gamma \delta \beta'}$ (this follows since we may assume that $\eta\subset U_\alpha$).  This shows that the counts there are ``untwisted". From this observation 
one can easily see that everything in the proof of exactness in \cite{OSzDouble} can go through to our twisted version.  
\end{proof}

In the above theorem, the cohomology classes $[\omega_r]$ are integral. In  practice one may need to use real cohomology class as well. In that situation, for a given homology class $[\omega] \in H^2(Y;\mathbb{R})$ we can express it as finite sum 
$$[\omega]=\sum a_iPD(\eta_i)$$
where $\eta_i$ are closed curves in the knot complement and $a_i \in \mathbb{R}$. Each $\eta_i$ can be viewed as a closed curve in $Y_r(K)$, so the expression $\sum a_iPD(\eta_i)$ also gives a real cohomology class in $Y_r(K)$, denote by $[\omega_r] \in H^2(Y_r(K);\mathbb{R})$. In the coborsim $W_r$, 
$$\sum a_i PD(\eta_i \times I)$$
is  a real cohomology class in $H^2(W_r;\mathbb R)$, hence  gives rise to homomorphism between $\omega$--twisted Floer homologies. With this understood, it is easy to see that  Theorem \ref{theorem:exact} still holds.

\begin{rem}The exact sequence in Theorem \ref{theorem:exact} depends on $\eta$, not just its Poincar\'e dual $[\omega] \in H^2(Y;\mathbb{Z})$. In fact if we take another closed curve $\eta'=\eta +k \cdot \mu$ (where $\mu$ is a meridian of $K$), this doesn't change $[\omega]$, but may change $[\omega_0], [\omega_1]$ and the  exact sequence.  For example, take $K \subset S^3$ to be the unknot  and $\eta=k \cdot \mu$ in the knot complement, then  $[\omega_0]$ is $k$ times the generator of  $H^2(S^2 \times S^1;\mathbb{Z})$. When $k \neq 0$, the corresponding exact sequence for the hat version is
\begin{displaymath}
\begin{xymatrix}{
\cdots \to \widehat{\underline{HF}}(S^3_0(K);\Lambda_{\omega_0})\ar[r] \ar[d]^\cong & \widehat{\underline{HF}}(S^3_1(K);\Lambda_{\omega_1}) \ar[r]\ar[d]^\cong &\widehat{\underline{HF}}(S^3;\Lambda_{\omega}) \ar[d]^\cong \to \cdots \\
 0\ar[r] & \Lambda \ar[r]^{1-t^k} & \Lambda }
\end{xymatrix}
\end{displaymath}
Clearly it depends on $k$. When $k=0$, the exact sequence is obtained from the corresponding exact sequence for untwisted Heegaard floer homology by tensoring with $\Lambda$.
\end{rem}

In \cite{OSzGenus}, Ozsv\'ath and Szab\'o used another version of twisted Floer homology $\underline{HF}^+(Y;[\omega])$, which is defined by using the $\mathbb{Z}[H^1(Y;\mathbb{Z})]$--module $\mathbb{Z}[\mathbb{R}]$. The $\omega$--twisted Floer  homology we used in this paper can be viewed as a completion of  $\underline{HF}^+(Y;[\omega])$. It is easy to see that there is a similar exact sequence in their context. More precisely, we have the following exact sequence:
 
\begin{equation}
\begin{xymatrix}{ \underline{HF}^+(Y;[\omega]) \ar[rr]
&&\underline{HF}^+(Y_0(K);[\omega_0]) \ar[dl]\\
&\underline{HF}^+(Y_1(K);[\omega_1])\ar[ul]&}
\end{xymatrix}
\end{equation}

With the above exact sequences in place, we can now prove Theorem \ref{theorem:main}.  We merely mimic the proof of Theorem 5.2 in \cite{OSzSympl}.
\begin{proof}[Proof of  Theorem \ref{theorem:main}]
For  a given cohomology class $[\omega]\in H^2(Y;\mathbb{Z})$ with $\omega(F)=d \neq 0$, choose a closed curve  $\eta \subset Y$ such that  its Poincar\'e dual $PD(\eta)$ equals $[\omega]$. Since the mapping class group of a torus  is generated as a monoid by right-handed Dehn twists along non-seperating curves, we can connect $Y$ to the three-manifold $S_0^3(T)$, which is obtained from $S^3$ by performing $0$--surgery on the right-handed trefoil, by a sequence of torus bundles
$$
\pi_i:Y^i \to S^1
$$
and cobordisms
\begin{displaymath}
\xymatrix{	Y=Y^0 \ar[r]^<<<<{W_0} & Y^1 \ar[r]^<<<<{W_1} & \cdots \ar[r]^<<<<{W_{n-1}} & Y^n = S_0^3(T)   }
\end{displaymath}
such that the monodromy of $Y^{i+1}$ differs from that of $Y^i$ by a single right-handed Dehn twist along a nonseparating  knot $K_i$ which lies in a fiber  $F_i$ of $\pi_i$.  The curve $\eta\subset Y$ induces curves $\eta_i\subset Y^i$ which can be assumed disjoint from $K_i$.  In this way, we get a sequence of cohomology classes $\omega_i=PD(\eta_i)\in H^2(Y^i;\mathbb{Z})$ such that $\omega_i(F_i)=d\neq0$.  The cobordism $W_i$ is obtained by attaching a single 2-handle to $Y^i \times I$ along the knot $K_i$  with framing $-1$ (with respect to the framing $K_i$ inherits from the fiber $F_i$).  Since $\eta_i$ is disjoint from $K_i$, $\eta_i\times I$ defines a relative homology class $[\eta_i\times I]\in H_2(W_i,\partial W_i;\mathbb{Z})$ and hence its Poincar\'e dual gives rise to homomorphisms between $\omega$--twisted Floer homologies:
$$\underline{F}^+_{W_i;PD(\eta_i \times I)}: \underline{HF}^+(Y^i;\Lambda_{\omega_i}) \to \underline{HF}^+(Y^{i+1};\Lambda_{\omega_{i+1}}).$$
We claim it is an isomorphism. Notice that $Y^{i+1} = (Y^i)_{-1}(K_i)$ where the 0-framing of $K_i$ is defined to be the framing $K_i$ inherits from the fiber, $F_i$.  Now consider $(Y^i)_0(K_i)$. This manifold contains  a  $2$--sphere $S_i$ (which is obtained from $F_i$ by surgering along $K_i$) and also an induced curve $\eta_i$ such that $\eta_i\cdot S_i=d\neq0$, therefore $\underline{HF}^+((Y^i)_0(K_i);\Lambda_{PD(\eta_i)})=0$ by proposition \ref{proposition:spheres}.  The exact sequence (\ref{Seq:01}) now proves the claim. 

This shows that $$\underline{HF}^+(Y;\Lambda_{\omega}) \cong \underline{HF}^+(S^3_0(T);\Lambda_{PD(\eta)})$$ where $\eta$ is the induced curve in $S_0^3(T)$.  We now identify the latter group.  For simplicity we write $\omega=PD(\eta)$. Identifying $\mathbb{Z}[H^1(S^3_0(T);\mathbb{Z})]$ with $\mathbb{Z}[t,t^{-1}]$,
 Ozsv\'ath and Szab\'o show in \cite{OSzAbsGr} that there is an identification of $\mathbb{Z}[t,t^{-1}]$--modules:
\[
	\underline{HF}^+_k(S_0^3(T);\mathbb{Z}[t,t^{-1}]) \cong \left\{
						\begin{array}{rl}
						\mathbb{Z} & \text{if } k\equiv-1/2 \,(\text{mod  }2)\,\text{and } k\geq-1/2\\
						\mathbb{Z}[t,t^{-1}] & \text{if } k=-3/2\\
						0 & \text{otherwise}
						\end{array}\right.
\]
Where the lefthand group is the universally twisted Heegaard Floer homology of $S^3_0(T)$, the $\mathbb{Z}$'s on the right are trivial $\mathbb{Z}[H^1(S^3_0(T);\mathbb{Z})]$--modules, and $\mathbb{Z}[t,t^{-1}]$ is a module over itself in the natural way.  By definition,
 $$\underline{CF}^+(S_0^3(T);\Lambda_{\omega}) = \underline{CF}^+(S_0^3(T);\mathbb{Z}[t,t^{-1}])\otimes_{\mathbb{Z}[t,t^{-1}]}\Lambda_{\omega}.$$
 We now apply the universal coefficients spectral sequence.  We need only compute $Tor^{\mathbb{Z}[t,t^{-1}]}_q(\mathbb{Z},\Lambda_{\omega})$.  We start with the free $\mathbb{Z}[t,t^{-1}]$-resolution of $\mathbb{Z}$:
\begin{displaymath}
\xymatrix{
	0 \ar[r] & \mathbb{Z}[t,t^{-1}] \ar[r]^{1-t} & \mathbb{Z}[t,t^{-1}]\ar[r] & \mathbb{Z} \ar[r] &0 }.
\end{displaymath}
Tensoring this complex over $\mathbb{Z}[t,t^{-1}]$ with $\Lambda_{\omega}$ and agumenting gives the complex 
\begin{displaymath}
\xymatrix{
	0 \ar[r] & \Lambda \ar[r]^{1-t^d} & \Lambda \ar[r] & 0 }.
\end{displaymath}
 Since we're working over $\Lambda$ and $d\neq0$, the middle map is an isomorphism and we see that $Tor^{\mathbb{Z}[t,t^{-1}]}_q(\mathbb{Z},\Lambda_{\omega})=0$ for all $q$.  Applying the universal coefficients spectral sequence, we obtain an isomorphism of $\Lambda$-modules $\underline{HF}^+(S_0^3(T);\Lambda_{\omega})\cong \Lambda$. Hence we have a $\Lambda$-module isomorphism
\[
	\underline{HF}^+(Y;\Lambda_{\omega})\cong \Lambda.
\]
Clearly $\underline{HF}^+(Y;\Lambda_{\omega})$ is supported in a single Spin$^c$ structure, since $\Lambda$ is a field.  Notice also that this Spin$^c$ structure must be torsion by the conjugation symmetry of $\underline{HF}^\circ$.  
\end{proof}  It is worth noting that alternate proofs of this theorem and proposition \ref{proposition:spheres} are possible through the use of inadmissible diagrams, which have been explored by Zhong Tao Wu.


\begin{thebibliography}{H}

\bibitem{AiN}{\bf Y Ai, Y Ni}, {\it Two applications of twisted Floer homology}, in preparation

\bibitem{Baldwin}{\bf J Baldwin}, {\it Heegard Floer homology and genus one, one boundary component open books}, preprint (2008), available at arXiv:0804.3624v1

\bibitem{CE}{\bf H Cartan,S Eilenberg}, {\it Homological algebra}, Reprint of the 1956 original. Princeton Landmarks in Mathematics. Princeton University Press, Princeton, NJ, 1999. xvi+390 pp.

\bibitem{Gh}{\bf P Ghiggini}, {\it Knot Floer homology detects genus-one fibered knots},
to appear in Amer. J. Math., available at arXiv:math.GT/0603445

\bibitem{JM}{\bf S Jabuka, T Mark}, {\it Product formulae for Ozsv\'ath--Szab\'o $4$--manifolds invariants},
to appear in Geom. Topol., available at arXiv:0706.0339

\bibitem{KM}{\bf P Kronheimer, T Mrowka}, {\it Monopoles and three-manifolds}, New Mathematical Monographs 10,
Cambridge University Press, Cambridge (2007)

\bibitem{KMOS}{\bf P Kronheimer, T Mrowka, P Ozsv\'ath, Z Szab\'o}, {\it Monopoles and lens space surgeries}, Ann. of Math. (2) 165 (2007), no. 2, 457--546

\bibitem{MD}{\bf D McDuff, D Salamon}, {\it $J$-holomorphic curves and symplectic topology}, American Mathematical Society Colloquium Publications, 52. American Mathematical Society, Providence, RI, 2004. xii+669 pp

\bibitem{Ni1}{\bf Y Ni}, {\it Knot Floer homology detects fibred
knots}, Invent. Math. 170 (2007), no. 3, 577--608

\bibitem{Ni2}{\bf Y Ni}, {\it Heegaard Floer homology and fibred $3$--manifolds},  preprint (2007), available at arXiv:0706.2031


\bibitem{OSzAnn1}{\bf P Ozsv\'ath, Z Szab\'o}, {\it Holomorphic disks and topological invariants for closed three-manifolds}, Ann. of Math.(2), 159 (2004), no. 3, 1027--1158


\bibitem{OSzAnn2}{\bf P Ozsv\'ath, Z Szab\'o}, {\it Holomorphic disks and  three-manifold invariants: properties and applications},
Ann. of Math.(2), 159 (2004), no. 3, 1159--1245


\bibitem{OSzSympl}{\bf P Ozsv\'ath, Z Szab\'o}, {\it Holomorphic triangle invariants and the topology of
symplectic four-manifolds}, Duke Math. J. 121 (2004), no. 1, 1--34

\bibitem{OSzAbsGr}{\bf P Ozsv\'ath, Z Szab\'o}, {\it Absolutely Graded Floer homologies and intersection
forms for four-manifolds with boundary}, Adv. Math. 173 (2003),
no. 2, 179--261

\bibitem{OSzKnot}{\bf P Ozsv\'ath, Z Szab\'o}, {\it Holomorphic disks and knot invariants},
Adv. Math. 186 (2004), no. 1, 58--116


\bibitem{OSzGenus}{\bf P Ozsv\'ath, Z Szab\'o}, {\it Holomorphic disks and genus bounds}, Geom. Topol. 8 (2004),
311--334 (electronic)

\bibitem{OSzDouble}{\bf P Ozsv\'ath, Z Szab\'o}, {\it On the Heegaard Floer homology of branched double-covers}. Adv. Math. 194 (2005), no. 1, 1--33


\bibitem{OSzFour}{\bf P Ozsv\'ath, Z Szab\'o}, {\it Holomorphic triangles and invariants for smooth four-manifolds},
Adv. Math. 202 (2006), no. 2, 326--400.

\bibitem{Ra}{\bf J Rasmussen}, {\it Floer homology and knot complements}, PhD Thesis, Harvard University (2003),
available at arXiv:math.GT/0306378


\end{thebibliography}
\end{document}